\documentclass[12pt]{extarticle}
\usepackage{extsizes}
\usepackage[reqno]{amsmath} 
\usepackage{amsmath}
\usepackage{amssymb}
\usepackage{amsthm}
\usepackage{amscd}
\usepackage{amsfonts}
\usepackage{graphicx}
\usepackage{dsfont}
\usepackage{fancyhdr}
\usepackage{enumerate}
\usepackage{youngtab}
\usepackage[utf8]{inputenc}
\usepackage{comment}

\usepackage{subfigure}
\usepackage[margin=1.5in]{geometry}
\usepackage{graphicx}
\usepackage{algorithm}
\usepackage{algpseudocode}
\usepackage{color}
\usepackage[hidelinks]{hyperref}
\usepackage{notation}

\theoremstyle{theorem}

\newtheorem{corollary}{Corollary}

\newtheorem{lemma}[corollary]{Lemma}
\newtheorem{lemma*}[lem6]{Lemma}

\newtheorem{proposition}[corollary]{Proposition}
\newtheorem{theorem}[corollary]{Theorem}

\newtheorem{problem}[corollary]{Problem}
\newtheorem*{conjecture*}{Conjecture}
\newtheorem{conjecture}{Conjecture}

\begin{document}

\AtEndDocument{%
  \par
  \medskip
  \begin{tabular}{@{}l@{}}%
    \textsc{Gabriel Coutinho}\\
    \textsc{Dept. of Computer Science} \\ 
    \textsc{Universidade Federal de Minas Gerais, Brazil} \\
    \textit{E-mail address}: \texttt{gabriel@dcc.ufmg.br} \\ \ \\
    \textsc{Thomás Jung Spier} \\
    \textsc{Dept. of Combinatorics \& Optimization} \\ 
    \textsc{University of Waterloo, Canada} \\
    \textit{E-mail address}: \texttt{tjungspi@uwaterloo.ca} \\ \ \\
     \textsc{Shengtong Zhang} \\
    \textsc{Dept. of Mathematics} \\ 
    \textsc{Stanford University, USA} \\
    \textit{E-mail address}: \texttt{stzh1555@stanford.edu}
  \end{tabular}}

\title{Conic programming to understand sums of squares of eigenvalues of graphs}
\author{Gabriel Coutinho\footnote{gabriel@dcc.ufmg.br} \and Thomás Jung Spier\footnote{tjungspi@uwaterloo.ca}\and Shengtong Zhang\footnote{stzh1555@stanford.edu}}
\date{\today}
\maketitle
\author
\vspace{-0.8cm}

\begin{abstract} 
   In this paper we prove a conjecture by Wocjan, Elphick and Anekstein (2018) which upper bounds the sum of the squares of the positive (or negative) eigenvalues of the adjacency matrix of a graph by an expression that behaves monotonically in terms of the vector chromatic number. One of our lemmas is a strengthening of the Cauchy-Schwarz inequality for Hermitian matrices when one of the matrices is positive semidefinite. 
   
   A related conjecture due to Bollob\'{a}s and Nikiforov (2007) replaces the vector chromatic number by the clique number and sums over the first two eigenvalues only. We prove a version of this conjecture with weaker constants. An important consequence of our work is a proof that for any fixed $r$, computing a rank $r$ optimum solution to the vector chromatic number semidefinite programming is NP-hard.   
   
   We also present a vertex weighted version of some of our results, and we show how it leads quite naturally to the known vertex-weighted version of the Motzkin-Straus quadratic optimization formulation for the clique number.
\end{abstract}

\begin{center}
\end{center}


\ \\ \

\section{Introduction}

Let $G$ be a graph, and $\aA$ be its adjacency matrix. In this paper we study the sums of squares of positive and negative eigenvalues of $\aA$. We adopt the convention that the eigenvalues of $\aA$ are indexed from largest to smallest, that is, the eigenvalues are $\lambda_1 \geq \dots \geq \lambda_n$.

Let $s^+$ and $s^-$ denote the sum of the squares of the positive and negative eigenvalues of $\aA$, respectively. Recently, $s^+$ and $s^-$ have attracted significant attention from the spectral graph theory community. See e.g. \cite{abiad2023, elphick2023symmetry}.

Our first result establishes a novel connection between $s^{\pm}$ and the vector chromatic number of $G$. Let $\overline{\aA}$ denote the adjacency matrix of the complement of $G$. We recall that the vector chromatic number of $G$, denoted by $\chi_{\mathrm{vec}}(G)$, is given by the semidefinite program (SDP) below (this SDP was originally defined independently in \cite{mceliece1977new,schrijver1979comparison}, but the nomenclature and geometric interpretation as a ``vector'' chromatic number is due to \cite{KMS})
\begin{align} \chi_{\mathrm{vec}}(G) =\hspace{2cm} \max \quad & \langle \jJ, \zZ\rangle, \label{sdp1}\\
\textrm{subject to} \quad & \langle \iI, \zZ\rangle = 1, \nonumber\\
& \zZ \circ \overline{\aA} = \0, \nonumber\\
& \zZ \geq \0, \, \zZ \succeq \0. \nonumber
\end{align}
Here, for a matrix $\zZ$ we write $\zZ \geq \0$ if all the entries of $\zZ$ are non-negative, and $ \zZ \succeq \0$ if $\zZ$ is positive semidefinite (PSD). A matrix $\zZ$ that satisfied both $\zZ \geq \0$ and $\zZ \succeq \0$ is called doubly nonnegative in literatue.  We use $\circ$ to denote the entry-wise product (also known as the Schur product or the Hadamard product) of matrices and $\langle \cdot, \cdot\rangle$ for the inner product $\langle \aA, \bB\rangle = \tr(\aA \bB)$ on the space of real, symmetric matrices of a fixed dimension. The vector chromatic number is an efficiently computable approximation of the chromatic number. It belongs to a hierarchy of analogs of the chromatic number
$$\omega(G) \leq \chi_{\mathrm{vec}}(G) \leq \vartheta(\bar{G}) \leq \chi_f(G) \leq \chi(G)$$
where $\omega, \vartheta, \chi_f, \chi$ denote the clique number, Lov\'{a}sz theta function, fractional chromatic number, and chromatic number of $G$ respectively.

\begin{theorem}\label{thm:main_result1} For every graph $G$ with $m$ edges, we have
\[
\min\{s^+,s^-\} \geq \frac{2m}{\chi_{\mathrm{vec}}(G)}.
\]
\end{theorem}
This theorem answers a conjecture raised by Wocjan, Elphick and Anekstein \cite{wocjan2018more}, and strengthens results due to Ando and Lin \cite{ando2015proof}  and Guo and Spiro \cite{guo2022new} in which $\chi_{\mathrm{vec}}$ is replaced by $\chi$ and $\chi_f$ respectively.

In fact, our proof shows a more general claim of independent interest. Let $\xX \succeq 0$ and $\yY$ be symmetric with $0$ diagonal. Let $\aA$ be the $(0,1)$-matrix whose support coincides with the support of $\xX \circ \yY$. Since $\yY$ has $0$ on the diagonal, $\aA$ is the adjacency matrix of some underlying graph $G$. The Cauchy-Schwarz inequality states that
\[
\langle \xX,\yY\rangle^2 \leq \langle\xX,\xX\rangle \langle\yY,\yY\rangle.
\]
We show a strengthening over this inequality.
\begin{lemma}
\label{lem:C-S-improved}
For any symmetric matrices $\xX$ and $\yY$ such that $\xX \succeq 0$ and $\yY$ has zero diagonal, define $G$ as above. We have
\begin{equation}
	\langle \xX,\yY\rangle^2 \leq \left(1 - \frac{1}{\chi_\mathrm{vec}(G)} \right)\langle \xX,\xX\rangle \langle \yY,\yY\rangle. \label{cs}
\end{equation}
\end{lemma}
This lemma could provide a big improvement over the usual Cauchy-Schwarz inequality when either $\xX$ or $\yY$ are sparse.
	
The inequality in Theorem \ref{thm:main_result1} can be rewritten as
\[
	\max\{s^+,s^-\} \leq \left(1 - \frac{1}{\chi_\mathrm{vec}(G)}\right) 2m.
\]
Under this formulation, it is straightforward to spot the parallel with a long-standing conjecture due to Bollob\'{a}s and Nikiforov \cite{bollobas2007cliques}. Let $\omega(G)$ denote the clique number of $G$. We note that $\omega(G) \leq \chi_\mathrm{vec}(G)$, but there exists graphs with bounded clique number and arbitrarily large $\chi_\mathrm{vec}(G)$. For example, Balla \cite[Corollary 9]{balla23} constructed a family of graphs $G_m$ with $m$ edges, $\omega(G_m) \leq \omega_0$ for some absolute constant $\omega_0$, and $\chi_{\mathrm{vec}}(G_m) \geq m^{1/3}$.

\begin{conjecture}[Bollob\'{a}s and Nikiforov \cite{bollobas2007cliques}]\label{conj:bollobas_nikiforov}
	If $G$ is not the complete graph, then
	\[\lambda_1^2 + \lambda_2^2 \leq \left(1 - \frac{1}{\omega(G)}\right) 2m.\]
\end{conjecture}
We observe the ``trivial" bound
$$\lambda_1^2 + \lambda_2^2 \leq \lambda_1^2 + \lambda_2^2 + \cdots + \lambda_n^2 = 2m.$$
So Bollob\'{a}s and Nikiforov's conjecture predicts that we can save a factor of $\frac{1}{\omega}$ over this bound.

This conjecture is motivated by a connection between spectral graph theory and Tur\'{a}n theory. A classical result of Nosal \cite{Nos1970} (see e.g. \cite{nikiforov2002some}) states
that if $G$ is triangle-free, then its spectral radius satisfies $\lambda_1 \le \sqrt{m}.$
Combined with the Rayleigh formula $\lambda_1 \geq 2m/n$, Nosal's theorem implies the foundational result of Mantel (see e.g. \cite{Bol1978}) that if $G$ is triangle-free, then $m\le \lfloor n^2/4\rfloor$.

In 2002, Nikiforov \cite{nikiforov2002some} proved the more general inequality
$$\lambda_1^2 \leq \left(1 - \frac{1}{\omega(G)}\right) 2m.$$
In particular, Nosal's theorem is the special case when $\omega(G) = 2$. By the Rayleigh formula, Nikiforov's theorem strengthens the celebrated Tur\'{a}n's theorem in extremal combinatorics. In 2007, Bollob\'{a}s and Nikiforov proposed their conjecture as a natural strengthening to this inequality.

Recently, this conjecture is emphasized in a list of open problems in spectral graph theory~\cite{liu23unsolved}. It has also been verified in a number of special cases. Ando and Lin \cite{ando2015proof} showed the conjecture for weakly perfect graphs, i.e. those with $\omega(G) = \chi(G)$; and Lin, Ning and Wu \cite{lin2021eigenvalues} showed it for graphs with $\omega(G)=2$, i.e. triangle-free graphs. Shengtong Zhang proved this conjecture for regular graphs \cite{zhang2024first}. Kumar and Pragada~\cite{kumar2024bollob} proved the conjecture for graphs with few triangles, in particular for planar graphs, book--free graphs and cycle--free graphs. Note also that Theorem~\ref{thm:main_result1} proves this conjecture for all non-complete graphs which satisfy $\omega(G) = \chi_{\mathrm{vec}}(G)$. Despite all these results for certain graph classes, there has not been a result applicable for all graphs which provides a saving term of the same order $\Omega_{\omega}(m)$ as Bollob\'{a}s and Nikiforov's Conjecture. For example, in light of the aforementioned construction of Balla, a direct application of Theorem~\ref{thm:main_result1} cannot show anything better than
$$\lambda_1^2 + \lambda_2^2 \leq 2m - \Omega_{\omega}(m^{2/3}).$$
In this paper, we show a weaker form of Conjecture~\ref{conj:bollobas_nikiforov} for all graphs. To our knowledge, this is the first general upper bound on $\lambda_1^2 + \lambda_2^2$ where the saving term over the trivial bound $2m$ is $\Omega_{\omega}(m)$.

\begin{theorem}\label{thm:main_result2}
Let
\[C = 1+\dfrac{\pi^2-4}{\pi^2+4}\approx 1.4231.\]
If $G$ is not the complete graph, then
	\[\lambda_1^2 + \lambda_2^2 \leq \left(1 - \frac{1}{C~\omega(G)}\right) 2m.\]
\end{theorem}
Furthermore, we are able to replace $C~\omega(G)$ with $(1 + o(1)) \omega(G)$, giving an asymptotic improvement for graphs with large clique numbers.
\begin{theorem}\label{thm:main_result3}
If $G$ is not the complete graph, then
	\[\lambda_1^2 + \lambda_2^2 \leq \left(1 - \frac{1}{\omega(G) + 50 \omega(G)^{5/6}}\right) 2m.\]
\end{theorem}

\subsection{Proof Strategy}

Our proof of Theorem~\ref{thm:main_result1} utilizes a carefully designed SDP \eqref{sdp2} for the vector chromatic number. We do not expect this semidefinite program for $\chi_{\mathrm{vec}}(G)$ to be unknown, but we could not find a reference for it. We substitute an appropriate matrix into this program to obtain our strengthening of the Cauchy-Schwarz inequality. Finally, we use a spectral decomposition argument similar to Ando-Lin's proof in \cite{ando2015proof} to conclude Theorem~\ref{thm:main_result1}.

To adapt this strategy to Theorem~\ref{thm:main_result2} and \ref{thm:main_result3}, our main idea is to add a low-rank restriction to the SDP \eqref{sdp2} defining $\chi_{\mathrm{vec}}(G)$. This suits our purpose because we are only interested in the top two eigenvalues of $G$. We realize the low-rank matrices in the programs as the Gram matrix of a family of low-dimensional vectors, and employ some geometrical arguments on these vectors. Thus, we show that the value of our modified SDP is a function of $\omega(G)$ (Lemma~\ref{lem:main2}), leading us to a weaker version of Theorem~\ref{thm:main_result2} with $C = 4$. To obtain the full Theorem~\ref{thm:main_result2} and \ref{thm:main_result3}, we consider further modified versions of the SDP introduced in \eqref{sdp2}, proving tighter bounds on the values of these modifications. We delay the precise definitions of these programs to Section~\ref{sec:bounded}.

Most of our work connects to the completely positive formulation of $\omega(G)$ due to De Klerk and Pasechnik \cite{klerkpasechnik}, in particular, it is easy to see that a rank $2$ constraint to the SDP in \eqref{sdp2} gives an equivalent optimization to it (see Proposition~\ref{prop:chivec2omega}). In Section~\ref{sec:vertexweights} we show how some of our results apply to the vertex-weighted version of $\omega(G)$, obtaining in a natural way the known vertex-weighted version of the Motzkin-Straus quadratic programming formulation, due to \cite{gibbons1997continuous}.

\subsection{Organization}
Section 2 is devoted to the proof of Theorem~\ref{thm:main_result1}. In Subsection~\ref{sec:bounded} we provide a short proof of Theorem~\ref{thm:main_result2} with the weaker constant $C=2$ . We prove it completely in Subsection~\ref{sec:constant} with more technical machinery. In Subsection~\ref{sec:sublinear}, we prove Theorem~\ref{thm:main_result3}. In Subsection~\ref{sec:counterexamples}, we discuss the limitations of our technique by proving some non-trivial lower bounds for our SDPs. In Subsection~\ref{sec:quickregular} we present a shorter proof for the regular graph case of the Bollobás-Nikiforov conjecture. Finally, in Section~\ref{sec:vertexweights}, we study vertex-weighted version of our results.

\subsection{Quick reference for notations}
\begin{enumerate}
    \item $G$ denotes a simple graph with vertex set $V$ and edge set $E$, $n$ denotes its number of vertices and $m$ its number of edges.
    \item $\aA$ is the adjacency matrix of $G$.
    \item $\ov{\aA}$ is the adjacency matrix of the complement of $G$.
    \item $\iI$ is the identity matrix, and $\jJ$ is the all-one matrix. 
    \item $\norm{\cdot}$ denotes the Frobenius norm $\lVert \mM \rVert = \sqrt{\tr(\mM^2)}$ on the space of real symmetric matrices of a given dimension. $\langle \cdot, \cdot\rangle$ stands for the corresponding inner product $\langle \aA, \bB\rangle = \tr(\aA \bB)$.
    \item For matrices $\xX$ and $\yY$ of the same dimension, we write $\xX \circ \yY$ as their entrywise product (also known as Schur product or Hadamard product).
    \item For a matrix $\xX$, we write $\xX \geq 0$ if all the entries of $\xX$ are non-negative, and $\xX \succeq 0$ if $\xX$ is a symmetric and positive semidefinite matrix.
\end{enumerate}
We suppress dependence on $G$ when there is no ambiguity.


\section{Wocjan-Elphick-Anekstein conjecture} \label{sec:2}

In this section we prove Theorem~\ref{thm:main_result1}.

Our proof for Theorem~\ref{thm:main_result1} is simpler than the argument due to Ando and Lin \cite{ando2015proof} which was used to prove a lower bound to the chromatic number, and which was also applied in Guo and Spiro \cite{guo2022new} to improve the result for the fractional chromatic number instead. We are able to devise a shortcut which ends up highlighting a strengthened version of the Cauchy-Schwarz inequality for a wide class of cases.

As before, recall that $\aA$ and $\overline{\aA}$ denote the adjacency matrices of $G$ and its complement. Observe that $\jJ = \iI + \aA + \overline{\aA}$.

\begin{lemma} \label{lem:chivec}
    Let $\zZ$ be a square matrix indexed by the vertices of a graph $G$. Assume $\zZ \geq \0$ and $\zZ \succeq \0$. Then
    \[
    \langle \aA,\zZ \rangle \leq \left(1-\dfrac{1}{\chi_{\mathrm{vec}}(G)}\right) \langle \jJ, \zZ \rangle.
    \]
\end{lemma}
\begin{proof}
First notice that
\[\langle \aA,\zZ \rangle \leq \left(1-\dfrac{1}{\chi_{\mathrm{vec}}(G)}\right) \langle \jJ, \zZ \rangle\iff \langle \jJ, \zZ\rangle\leq \chi_{\mathrm{vec}}(G) \langle \iI + \overline{\aA}, \zZ\rangle.\]

We claim that the inequality above holds for every $\zZ\geq \0$, $\zZ\succeq \0$. This follows from the stronger claim that $\chi_{\mathrm{vec}}(G)$ is equal to the following semidefinite program:

\begin{align} \max \quad & \langle \jJ, \zZ\rangle \label{sdp2}\\
\textrm{subject to} \quad & \langle \iI+ \overline{\aA}, \zZ \rangle = 1 \nonumber \\
& \zZ\geq \0, \, \zZ\succeq \0. \nonumber
\end{align}
Every feasible solution of \eqref{sdp1} is clearly a feasible solution for \eqref{sdp2} with the same objective value. Thus $\chi_{\mathrm{vec}}(G)$ is smaller than or equal to the optimum of \eqref{sdp2}.

On the other hand, if $\zZ_{0}$ is an optimum solution for \eqref{sdp2}, let $S$ denote the set of edges $uv$ in $E(\overline{G})$ for which $(\zZ_0)_{uv} > 0$, and $e_u$ the characteristic vector of vertex $u$. Then the perturbation of $\zZ_0$ defined by
\[
\widetilde{\zZ}_0:=\zZ_0 + \displaystyle\sum_{uv\in S}(\zZ_0)_{uv}(\ee_u-\ee_v)(\ee_u-\ee_v)^\T
\]
satisfies
\[
\widetilde{\zZ}_0 \geq \0, \ \widetilde{\zZ}_0 \succeq \0, \  \widetilde{\zZ}_0 \circ \overline{\aA} = \0, \ \langle \widetilde{\zZ}_0,\iI\rangle =  \langle \zZ_0,\iI+\overline{\aA}\rangle = 1,\ \text{and} \ \langle \widetilde{\zZ}_0,\jJ\rangle =  \langle \zZ_0,\jJ\rangle,
\]
therefore $\widetilde{\zZ}_0$ is a feasible solution for \eqref{sdp1} with the same objective value as $\zZ_0$ in \eqref{sdp2}.
\end{proof}
Using this SDP, we prove our strengthening of the Cauchy-Schwarz inequality.
\begin{proof}[Proof of Lemma~\ref{lem:C-S-improved}]
Denote $\overline{\xX} = \xX \circ \aA$ and $\overline{\yY} = \yY \circ \aA$. Since $\aA$ has the same support as $\xX \circ \yY$, we have $\langle \xX, \yY\rangle = \langle \ov{\xX}, \ov{\yY}\rangle$. By the usual Cauchy-Schwarz inequality for matrices, we have 
\[\langle \xX, \yY\rangle^2 = \langle \ov{\xX}, \ov{\yY}\rangle^2 \leq \langle \ov{\xX},  \ov{\xX} \rangle \langle \ov{\yY}, \ov{\yY} \rangle.\]
The first term can be simplified as
\[\langle \ov{\xX},  \ov{\xX} \rangle = \langle \aA\circ \xX,\aA \circ \xX\rangle= \langle \aA,\xX \circ \xX\rangle.\] 
Clearly, we have $\xX\circ \xX\geq \0$. Furthermore, by Schur's product theorem, we have $\xX\circ \xX\succeq \0$. By Lemma~\ref{lem:chivec} we also have
\[\langle \aA, \xX\circ \xX\rangle\leq \left(1-\dfrac{1}{\chi_{\mathrm{vec}}(G)}\right)\langle \jJ,\xX \circ \xX\rangle = \left(1-\dfrac{1}{\chi_{\mathrm{vec}}(G)}\right)\langle \xX,\xX\rangle.\]
Note also that $\langle \ov{\yY}, \ov{\yY} \rangle \leq \langle \yY, \yY \rangle$.
The lemma follows by combining the three inequalities. 
\end{proof}
In particular, we take $\yY = \aA$ in the lemma. Observe that $\langle \aA, \aA \rangle = 2m$. Thus, we obtain the following graph-theoretic corollary.
\begin{corollary}\label{cor:CS_chivec} Let $G$ be a graph and $\xX$ be any PSD matrix. Then, 
\[\langle \aA, \xX\rangle^2\leq \left(1-\dfrac{1}{\chi_{\mathrm{vec}}(G)}\right)2m\,\langle \xX, \xX\rangle. \]
\end{corollary}

Note also that the same conclusion of the corollary holds if we instead choose $Y=\hat\aA$ in the Lemma~\ref{lem:C-S-improved} for a signed adjacency matrix $\hat\aA$.

Theorem~\ref{thm:main_result1} now follows by a spectral decomposition argument. Recall that $\aA$ admits a spectral decomposition
$$\aA = \sum_{\lambda} \lambda \eE_{\lambda}$$
where $\lambda$ ranges over the eigenvalues of $\aA$, and $\eE_{\lambda}$ is the projection matrix onto the $\lambda$-eigenspace of $\aA$. To study the sums of squares of positive and negative eigenvalues, we introduce the Jordan decomposition~\cite{BhatiaMatrixAnalysis}
$$\aA = \aA^+ - \aA^-$$
where $\aA^+$ and $\aA^-$ are PSD matrices defined by
\[
\aA^+ = \sum_{\lambda > 0} \lambda \eE_\lambda \quad \text{and} \quad \aA^- = - \sum_{\lambda< 0} \lambda \eE_\lambda.
\]
Then we have
\[
s^+ = \lVert \aA^+ \rVert^2 \quad \text{and} \quad s^- = \lVert \aA^- \rVert^2.
\]

\begin{proof}[Proof of Theorem \ref{thm:main_result1}] We apply Corollary~\ref{cor:CS_chivec} to $\xX=\aA^+$. Observing that $\aA\aA^+=(\aA^+)^2$, we have
\[\langle \aA^+, \aA^+\rangle^2=\langle \aA, \aA^+\rangle^2\leq \left(1-\dfrac{1}{\chi_{\mathrm{vec}}(G)}\right)2m\,\langle \aA^+, \aA^+\rangle,\]
so 
\[s^+=\langle \aA^+, \aA^+\rangle\leq \left(1-\dfrac{1}{\chi_{\mathrm{vec}}(G)}\right)2m.\]
Recalling that $s^+ + s^- = 2m$, it follows that:
\[\dfrac{2m}{\chi_{\mathrm{vec}}(G)} \leq s^-.\]
The proof for $s^+$ is analogous and obtained upon making $\xX=\aA^-$.
\end{proof}

Again, it is worthwhile to notice that we may derive a similar result for signed graphs.

\begin{corollary}
	For every graph with $m$ edges, and every signed adjacency matrix $\hat \aA$, if $s^+$ and $s^-$ are respectively the sums of the squares of the positive and the negative eigenvalues of $\hat \aA$, then
	\[
		\min\{s^+,s^-\} \geq \frac{2m}{\chi_{\mathrm{vec}}(G)}.
	\]
\end{corollary}


\section{Bollob\'{a}s-Nikiforov conjecture} \label{sec:3}

In this section, we study Conjecture~\ref{conj:bollobas_nikiforov}.

Our starting point is a completely positive program whose value is the clique number $\omega(G)$. The cone of completely positive matrices is defined as
\[
	\mathcal{CP}_n := \left\{\mM \in \Rds^{n\times n}: \mM = \sum_{i = 1}^k \xx_i \xx_i^\T , \xx_i \in \Rds^n , \xx_i \geq \0\right\}
\]

Extending the celebrated Motzkin-Straus inequality \cite{MotzkinStraus}, De Klerk and Pasechnik \cite{klerkpasechnik} proved that
\begin{align} \omega(G) =\hspace{2cm} \max \quad & \langle \jJ, \zZ\rangle, \label{sdp1}\\
\textrm{subject to} \quad & \langle \iI+\ov{\aA}, \zZ\rangle = 1, \nonumber\\
& \zZ\in \mathcal{CP}_n. \nonumber
\end{align}
The resemblance with the SDP \eqref{sdp2} is evident: the set of doubly nonnegative matrices forms a cone that contains the cone $\mathcal{CP}_n$, so we naturally manifest the inequality $\omega(G) \leq \chi_{\mathrm{vec}}(G)$. At this point it is worth pointing out that replacing the cone by the cone of positive semidefinite matrices \textsl{does not} make the optimum equal to $\vartheta(\ov{G})$. In fact, the following modification of program \eqref{sdp1} 
\begin{align} \hspace{2cm} \max \quad & \langle \jJ, \zZ\rangle, \nonumber \\
\textrm{subject to} \quad & \langle \iI+\ov{\aA}, \zZ\rangle = 1, \nonumber\\
& \zZ \succeq 0. \nonumber
\end{align}
is unbounded unless all the eigenvalues of $\ov{\aA}$ are at least $-1$, which occurs only if $\bar{G}$ is a disjoint union of complete graphs. 

A simple manipulation (as was done in Lemma~\ref{lem:chivec}) leads to the following inequality for all $\zZ \in \mathcal{CP}_n$:
\begin{align} \label{eq:ms}
	\langle \aA,\zZ \rangle \leq \left(1-\dfrac{1}{\omega(G)}\right) \langle \jJ, \zZ \rangle.
\end{align}

In the proof of Corollary~\ref{cor:CS_chivec}, we used the fact that if $\xX \succeq \0$ then $\xX \circ \xX$ is doubly nonnegative. Here we use the fact that if $\xX \in \mathcal{CP}_n$, then $\xX \circ \xX \in \mathcal{CP}_n$. This is quite easy to see, as $(\xx\xx^\T) \circ (\yy\yy^\T) = (\xx \circ \yy)(\xx \circ \yy)^\T$. Thus, we obtain an analog of Corollary~\ref{cor:CS_chivec} for completely positive matrices.

\begin{corollary}\label{cor:CS_omega} Let $G$ be a graph and assume that $\xX \in \mathcal{CP}_n$. Then, 
\[\langle \aA, \xX\rangle^2\leq \left(1-\dfrac{1}{\omega(G)}\right)2m\,\langle \xX, \xX\rangle. \]
\end{corollary}

For example, let $\eE_0$ be the projector onto the eigenspace corresponding to the largest eigenvalue $\lambda_1$ of $\aA$ (in particular, it is completely positive). Choosing $\xX = \eE_0$ recovers Nikiforov's spectral Tur\'{a}n inequality \cite{nikiforov2002some}
$$\lambda_1^2 \leq \left(1-\dfrac{1}{\omega(G)}\right)2m.$$

\subsection{The bounded rank vector chromatic numbers} \label{sec:bounded}

As mentioned in the introduction, our strategy is to modify the SDP \eqref{sdp2} by adding rank-restrictions to the matrix $\zZ$. We now formally define the \textit{bounded rank vector chromatic number} $\chi_{\mathrm{vec,d}}$ by the program
\begin{align} \chi_{\mathrm{vec,d}}(G) := \hspace{2cm} \max \quad & \langle \jJ, \zZ\rangle, \label{sdp5}\\
\textrm{subject to} \quad & \langle \iI+\ov{\aA}, \zZ\rangle = 1, \nonumber\\
& \zZ\geq \0, \, \zZ\succeq \0, \nonumber \\& \mathrm{rank}~\zZ \leq d  \nonumber.
\end{align}
Note that \eqref{sdp2} gives $\chi_{\mathrm{vec,n}}(G) = \chi_{\mathrm{vec}}(G)$ and \eqref{sdp1} implies $\chi_{\mathrm{vec,1}}(G) = \omega(G)$, because any doubly nonnegative matrix of rank $1$ must be completely positive.

In addition, we define two modifications of this program for $d = 3$. The first modification is 
\begin{align} 
\chi_{\mathrm{vec}, 3}'(G) :=\hspace{2cm} \max \quad & \langle \jJ, \xX\circ\xX \rangle, \label{sdp5-mod1} \\
\textrm{subject to} \quad & \langle \iI+\overline{\aA}, \xX\circ \xX\rangle = 1, \nonumber\\& \xX \succeq \0, \nonumber
\\& \mathrm{rank}~\xX \leq 2. \nonumber
\end{align}
Note that $``3"$ in $\chi_{\mathrm{vec}, 3}'(G)$ is not a typo: if $\xX$ has rank $2$, then $\xX \circ \xX$ has rank at most $3$, so we are in fact looking at feasible solutions to the program defining $\chi_{\mathrm{vec}, 3}(G)$ which are Schur squares of rank $2$ matrices.

The second modification adds one more restriction to the first modification. We write $\xX_G \geq 0$ if $\xX_{ij} \geq 0$ for any edges $\{i, j\} \in E(G)$.
\begin{align} 
\chi''_{\mathrm{vec}, 3}(G) :=\hspace{2cm} \max \quad & \langle \jJ, \xX\circ\xX \rangle, \label{sdp5-mod2} \\
    \textrm{subject to} \quad & \langle \iI+\overline{\aA}, \xX\circ \xX\rangle = 1, \nonumber\\& \xX \succeq \0, \ \xX_G \geq 0,\nonumber
    \\& \mathrm{rank}~\xX \leq 2. \nonumber
\end{align}

Since each program narrows the feasible region of the previous one, we have
$$\chi''_{\mathrm{vec}, 3}(G) \leq \chi_{\mathrm{vec}, 3}'(G)  \leq \chi_{\mathrm{vec}, 3}(G).$$
The next lemma shows how upper bounds on each of these programs translate to an upper bound on $\lambda_1^2 + \lambda_2^2$.
\begin{lemma}
\label{lem:sdp-to-BN}
Suppose $f: \nN \to \mathbf{R}^+$ is a function such that for any graph $G$ with $\omega(G) \leq \omega$, we have $\chi''_{\mathrm{vec}, 3}(G) \leq f(\omega)$. Then for any graph $G$ that is not the complete graph, we have
$$\lambda_1^2 + \lambda_2^2 \leq \left(1 - \frac{1}{f(\omega(G))}\right) 2m.$$
\end{lemma}
Since $\chi''_{\mathrm{vec}, 3}(G) \leq \chi_{\mathrm{vec}, 3}'(G)  \leq \chi_{\mathrm{vec}, 3}(G)$, we can replace $\chi''_{\mathrm{vec}, 3}(G) \leq f(\omega)$ with $\chi_{\mathrm{vec}, 3}(G) \leq f(\omega)$ or $\chi'_{\mathrm{vec}, 3}(G) \leq f(\omega)$ in the statement.
\begin{proof}
Since $G$ is not complete, we have $\lambda_1 \geq \lambda_2 \geq 0$. We define a rank $2$ matrix
$$\xX = \lambda_1 (\vv_1 \vv_1^\T) + \lambda_2 (\vv_2 \vv_2^\T).$$ 
We consider the subgraph $G'$ of $G$ on the same vertex set whose edges are $\{i, j\} \in E(G)$ with $\xX_{ij} \geq 0$. Let $\aA'$ be the adjacency matrix of $G'$. Then $\xX_{G'} \geq 0$ and $\omega(G') \leq \omega(G)$. By the program defining $\chi''_{\mathrm{vec}, 3}(G')$, we have
$$\langle \jJ, \xX\circ\xX \rangle \leq \chi''_{\mathrm{vec}, 3}(G') \langle \iI + \ov{\aA'}, \xX\circ\xX \rangle \leq f(\omega(G)) \langle \iI + \ov{\aA'}, \xX\circ\xX \rangle.$$
Recalling that $\jJ = \iI + \ov{\aA'} + \aA'$, we obtain
$$\langle \aA', \xX\circ\xX \rangle \leq \left(1 - \frac{1}{f(\omega(G))}\right)\langle \jJ, \xX\circ\xX \rangle.$$
Furthermore, we have
$$\lambda_1^2 + \lambda_2^2 = \langle \aA, \xX \rangle \leq \langle \aA', \xX \rangle.$$
Therefore we conclude that
\begin{align*}
     (\lambda_1^2 + \lambda_2^2)^2 = \langle \aA,\xX \rangle^2 \leq \langle \aA', \xX \rangle^2 & \leq \langle \aA',\aA' \rangle \langle \aA' , \xX \circ \xX \rangle \\ & \leq 2m ~ \left(1 - \frac{1}{f(\omega(G))}\right)  \langle \jJ,\xX \circ \xX \rangle \\& = 2m ~ \left(1 - \frac{1}{f(\omega(G))}\right)  (\lambda_1^2 + \lambda_2^2)
\end{align*}
as desired.
\end{proof}
We now state our main results on $\chi_{\mathrm{vec,3}}(G)$ and its two variants. In the next three subsections, we will prove the following successively stronger upper bounds on the three programs. 
\begin{lemma}\label{lem:chivec_line_weak}
    For any graph $G$, we have
    $$ \chi_{\mathrm{vec,3}}(G) \leq 4 \omega(G).$$
\end{lemma}
\begin{lemma}\label{lem:chivec_line} For any graph $G$, we have
\[
\chi_{\mathrm{vec}, 3}'(G)\leq C~\omega(G).
\]
where recall that $C = 1 + \frac{\pi^2 - 4}{\pi^2 + 4} \approx 1.4231$.
\end{lemma}
\begin{lemma}
    \label{lem:chi-vec-doubleprime-upper}
    For any graph $G$, we have
    $$\chi''_{\mathrm{vec}, 3}(G) \leq \omega(G) + 50 \omega(G)^{5/6}.$$
\end{lemma}
Combined with Lemma~\ref{lem:sdp-to-BN}, we observe that Lemma~\ref{lem:chivec_line_weak} implies Theorem~\ref{thm:main_result2} with the weaker constant $C = 4$, Lemma~\ref{lem:chivec_line} implies Theorem~\ref{thm:main_result2}, and Lemma~\ref{lem:chi-vec-doubleprime-upper} implies Theorem~\ref{thm:main_result3}. Thus, we complete the proofs of Theorem~\ref{thm:main_result2} and Theorem~\ref{thm:main_result3}.

Given these upper bounds, it is natural to wonder if there is hope of resolving the Bollob\'{a}s-Nikiforov conjecture by tightening the bounds on $\chi_{\mathrm{vec}, 3}(G)$, $\chi_{\mathrm{vec}, 3}'(G)$ and $\chi_{\mathrm{vec}, 3}''(G)$. Indeed, if we can show that any of them is at most $\omega(G)$ for all graphs $G$, then the Bollob\'{a}s-Nikiforov conjecture follows from Lemma~\ref{lem:sdp-to-BN}.

Unfortunately, we show that this is not the case for $\chi_{\mathrm{vec}, 3}(G)$ and $\chi_{\mathrm{vec}, 3}'(G)$. One way to observe this is that Lemma~\ref{lem:sdp-to-BN} also holds for signed graphs with essentially the same proof as we have in this paper. However Conjecture~\ref{conj:bollobas_nikiforov} does not, as was shown in~\cite{kannan2023signed} with the example $G = C_5$.

In Section~\ref{sec:counterexamples} we will exhibit graphs $G$ with arbitrarily large clique number such that both $\chi_{\mathrm{vec}, 3}(G)$ and $\chi_{\mathrm{vec}, 3}'(G)$ are bounded away from $\omega(G)$.
\begin{proposition}
    \label{prop:chi-3-lower}
    For any even $\omega$, there exists a graph $G$ with $\omega(G) = \omega$ such that $\chi_{\mathrm{vec}, 3}(G) \geq \chi_{\mathrm{vec}, 3}'(G) > 1.04\omega$.
\end{proposition}
This result shows that we cannot prove Conjecture~\ref{conj:bollobas_nikiforov} by tightening our bounds on either  $\chi_{\mathrm{vec}, 3}(G)$ or $\chi_{\mathrm{vec}, 3}'(G)$. In fact, this serves as our motivation for introducing $\chi_{\mathrm{vec}, 3}''(G)$. Lemma~\ref{lem:chi-vec-doubleprime-upper} shows that Proposition~\ref{prop:chi-3-lower} does not apply to $\chi_{\mathrm{vec}, 3}''(G)$ and we have $\chi_{\mathrm{vec}, 3}''(G) \leq \omega(G) + o(\omega(G))$. Thus, we still have reason to hope that tightening our bound on $\chi_{\mathrm{vec}, 3}''(G)$ can resolve the Bollob\'{a}s-Nikiforov conjecture.

\subsection{Proof of Lemma~\ref{lem:chivec_line_weak}}
Before we move forward, the following proposition is a good motivation for the results to come.

\begin{proposition} \label{prop:chivec2omega}
    For any graph $G$, we have
    \[
        \chi_{\mathrm{vec,2}}(G) = \omega(G).
    \]
\end{proposition}
\begin{proof}
    This follows immediately from the nontrivial yet well-known fact (see e.g. \cite[Theorem 3.5]{shaked2021copositive}) that if $\mM$ is doubly nonnegative of rank $2$, then $\mM$ is completely positive. For the reader's convenience, we include a quick proof of this fact. If an $n \times n$ matrix $\mM$ is doubly nonnegative and has rank $2$, then there exists an $n \times 2$ matrix $\bB$ such that
\[
	\mM = \bB \bB^\T.
\]
The inner product between any two rows of $\bB$ must be $\geq \0$ because $\mM \geq \0$, so the cone generated by the rows is contained in a polyhedral cone whose extreme rays are at an angle of at most $\pi /2$. Let $\uU$ be the rotation in $\Rds^2$ that places the two rays within $\Rds^2_+$. Then we have
\[
	\mM= (\uU\bB^\T)^\T (\uU\bB^\T),
\]
and $\uU\bB^\T \geq \0$, so $\mM$ is completely positive.
\end{proof}

We now proceed to show that $\chi_{\mathrm{vec,d}}(G)$ is bounded by $\omega(G)$ times a function of $d$. This is in contrast with $\chi_{\mathrm{vec}}(G)$, which is not bounded by any function of $\omega(G)$. The $d = 3$ case of this lemma implies Lemma~\ref{lem:chivec_line_weak}.

\begin{lemma}\label{lem:main2}
    Let $d$ be fixed, and let $G$ be a graph. Assume $\zZ$ is a feasible solution for \eqref{sdp5} which is the Gram matrix of vectors contained in $C$ orthants. Then we have
    \[
        \langle \jJ,\zZ \rangle \leq C~\omega(G).
    \]
    In particular, $C$ is always upper bounded by $2^{d-1}$, so for any graph $G$ we have
    \[\omega(G) \leq \chi_{\mathrm{vec,d}}(G) \leq 2^{d-1} \omega(G).\]
\end{lemma}

\begin{proof}
    The lower bound follows since $\chi_{\mathrm{vec,d}}(G) \geq \chi_{\mathrm{vec,1}}(G) = \omega(G)$, where the first inequality is due to narrower feasible region. We now prove the upper bound.

	Assume $\zZ \geq \0$ is the Gram matrix of vectors $\{\zz_i\}_{i = 1}^n$, $\zz_i \in \Rds^d$. Upon applying a rotation to these vectors, we may assume that a nonzero vector, say $\zz_1$, is parallel to $\ee_1$. All other rotated vectors belong to the hyperspace $\{\zz : \langle \zz , \ee_1 \rangle \geq 0\}$, and this is entirely contained in $2^{d-1}$ orthants. Applying an orthogonal transformation to a set of vectors does not change their Gram matrix, so we may assume that $\zZ$ is the Gram matrix of vectors contained in $C(d)$ orthants, with $C(d) \leq 2^{d-1}$. 
	
	We partition the set $\{\zz_i\}$ into $C$ subsets, each entirely contained in an orthant. We denote by $\zZ_j$ the $n\times n$ Gram matrix of vectors belonging to the $j$th part, replacing the remaining vectors by the $\0$ vector. It follows that each $\zZ_j$ is completely positive (easily verified upon rotating its corresponding orthant to the nonnegative orthant), thus formulation \eqref{sdp1} gives that
	\[
		\omega(G) \langle \iI+\ov{\aA} , \zZ_j \rangle \geq \langle \jJ , \zZ_j\rangle, 
	\]
	for all $j$. As $\zZ \geq \sum \zZ_j$, it also follows that
	\[
		\langle \iI+\ov{\aA} , \zZ \rangle \geq \sum_{j =1}^{C} \langle \iI+\ov{\aA} , \zZ_j \rangle.
	\]
	On the other hand, $\zZ = \bB\bB^\T$, $\bB = \sum \bB_j$, and $\zZ_j = \bB_j\bB_j^\T$. Also $\langle \zZ,\jJ \rangle = \langle \bB^\T \1, \bB^\T \1\rangle$. Thus, by Cauchy-Schwarz, it follows that
	\[
		\langle \zZ,\jJ \rangle \leq C  \sum_{j =1}^{C} \langle \jJ, \zZ_j \rangle,
	\]
	whence the result follows.
\end{proof}
In particular, Lemma~\ref{lem:chivec_line_weak} follows from the $d = 3$ case. Finally, Lemma~\ref{lem:main2} also implies the following corollary on the computational hardness of the low-rank program \eqref{sdp5}.

\begin{corollary}
	For any fixed $d$, unless P=NP, $\chi_{\mathrm{vec,d}}(G)$ cannot be computed in polynomial time.
\end{corollary}
\begin{proof}
	It is well known that approximating \texttt{MAXCLIQUE} within a constant factor is NP-hard (see \cite{haastad1999clique}).
\end{proof}

\subsection{Proof of Lemma~\ref{lem:chivec_line}} \label{sec:constant}

As motivation for the proof of Lemma~\ref{lem:chivec_line}, we present a weaker (and simple) version with $C = 2$ using the strategy of Lemma~\ref{lem:main2}. This result implies Theorem~\ref{thm:main_result2} with $C = 2$.
\begin{lemma}
    For every graph $G$, it holds that
\[
\chi_{\mathrm{vec}, 3}'(G)\leq 2~\omega(G).
\]
\end{lemma}
\begin{proof}
Let $\xX$ be any feasible solution for \eqref{sdp5-mod1}. Since $\xX$ is PSD with rank $2$, there exists vectors $\vv_1$ and $\vv_2$ such that
	\[
		\xX = \vv_1\vv_1^\T + \vv_2 \vv_2^\T.
	\]
	Note that
	\[
		\xX \circ \xX = (\vv_1 \circ \vv_1)(\vv_1 \circ \vv_1)^\T + 2(\vv_1 \circ \vv_2)(\vv_1 \circ \vv_2)^\T + (\vv_2 \circ \vv_2)(\vv_2 \circ \vv_2)^\T
	\]
	This is a doubly nonnegative matrix of rank $3$. In general, Lemma~\ref{lem:main2} applies as this is the Gram matrix of vectors contained in at most four orthants, but in this case we can do better. Let $\bB$ be the $3$-column matrix
	\[
		\bB = \left(\begin{array}{c|c|c} (\vv_1 \circ \vv_1) & \sqrt{2} (\vv_1 \circ \vv_2) & (\vv_2 \circ \vv_2)\end{array}\right).
	\]
	Note that $\bB\bB^\T = \xX \circ \xX$, so $\xX \circ \xX$ is the Gram matrix of the rows of $\bB$. Each row of $\bB$ contains nonnegative entries in the first and the third coordinates, thus all rows of $\bB$ are contained in two orthants. Therefore, Lemma~\ref{lem:main2} with $\zZ = \xX \circ \xX$ and $C = 2$ gives that
        \[
            \langle \jJ,\xX \circ \xX \rangle \leq 2~\omega(G)~\langle \iI+\ov{\aA} , \xX \circ \xX \rangle = 2~\omega(G),
        \]
    as desired.
\end{proof}

Let us improve upon the constant. Define \[C:=1+\dfrac{\pi^2-4}{\pi^2+4}=\dfrac{2\pi^2}{\pi^2+4}\approx 1.4231.\] Assume $\xX \succeq \0$ is any PSD matrix of rank $2$, whose rows and columns are indexed by the vertex set $V$ of $G$. Decompose $\xX$ as a Gram-Schmidt matrix $\xX = \vV \vV^T$ for $\vV \in \rR^{V \times 2}$. For
each vertex $i \in V$ let $\vv_i \in \rR^2$ be the row of $\vV$ corresponding to $i$. Let $r_i = \norm{\vv_i}$ and $\theta_i = \arg \vv_i \in \rR / 2\pi \zZ$. Then we have
$$\vv_i = r_i \e^{\ii \theta_i} \text{ and } \xX_{i, j} = r_i r_j \cos(\theta_i - \theta_j).$$

To analyze the geometry of the vectors $\vv_i$, we introduce the following notations.  For a given $\alpha$ in $\rR / 2\pi \zZ$ let $P_+^\alpha$ and $P_-^\alpha$ be the set of vertices $i$ with $\theta_i$ in $I_+^\alpha=[\alpha,\alpha+\pi/2)\sqcup [\alpha+\pi,\alpha+3\pi/2)$ and $I_-^\alpha=(\alpha,\alpha-\pi/2]\sqcup (\alpha-\pi,\alpha-3\pi/2]$, respectively. Let $\vV_+^\alpha$ and $\vV_-^\alpha$ be the Gram matrices corresponding to the sets of vectors $\{\vv_i\}_{i\in P_+^\alpha}$ and $\{\vv_i\}_{i\in P_-^\alpha}$, respectively, and write $\xX_+^\alpha = \vV_+^\alpha(\vV_+^\alpha)^T$ and $\xX_-^\alpha = \vV_-^\alpha(\vV_-^\alpha)^T$. As in the proof of Lemma~\ref{lem:main2}, note that $\xX_+^\alpha \circ \xX_+^\alpha$ and $\xX_-^\alpha \circ \xX_-^\alpha$ are both completely positive.

The key insight leading to the proof of Lemma~\ref{lem:chivec_line} is the following result.

\begin{lemma}\label{lem:partition} There exists some $\alpha$ in $\rR / 2\pi \zZ$ such that
\[
\langle \jJ, \xX\circ \xX \rangle \leq C(\langle \jJ, \xX_+^\alpha\circ \xX_+^\alpha \rangle+\langle \jJ, \xX_-^\alpha \circ \xX_-^\alpha \rangle).
\]
\end{lemma}

We show how Lemma~\ref{lem:chivec_line} follows from Lemma~\ref{lem:partition}.

\begin{proof}[Proof of Lemma~\ref{lem:chivec_line}]
Note that $\xX_+^\alpha\circ \xX_+^\alpha$ and $\xX_-^\alpha\circ \xX_-^\alpha$ are disjoint principal submatrices of $\xX \circ \xX$ which partition its diagonal, and they are both completely positive.  Hence Lemma~\ref{lem:partition} and \eqref{sdp1} imply that
\begin{align*}
\langle \jJ, \xX\circ \xX \rangle ~&\leq~ C~\omega(G)~ \left\langle \iI + \ov{\aA} , \pmat{\xX_+^\alpha\circ \xX_+^\alpha & \0 \\ \0 & \xX_-^\alpha\circ \xX_-^\alpha}\right\rangle ~\\&\leq~ C~\omega(G)~ \left\langle \iI + \ov{\aA} , \xX \circ \xX \right\rangle,
\end{align*}
as we wanted.
\end{proof}
Let $\mu$ be the measure in $\rR / 2\pi \zZ$ formed by a finite number of atoms
$$\mu := \sum_{i \in V} r_i^2 \delta_{\theta_i}.$$ 
Observe that, 
\begin{align*}
\langle \jJ, \xX_+^\alpha \circ \xX_+^\alpha \rangle&=\sum_{i,j\in P_+^\alpha} \langle \vv_i, \vv_j \rangle^2 =  \sum_{i,j\in P_+^\alpha} r_i^2 r_j^2 \cos^2(\theta_i-\theta_j)\\&=\displaystyle\int\int_{I_+^\alpha\times I_+^\alpha}\cos^2(\theta-\tilde{\theta})d\mu(\theta)d\mu(\tilde{\theta})\\&=\dfrac{1}{2}\left(\displaystyle\int_{I_+^\alpha}d\mu(\theta)\right)^2+\dfrac{1}{2}\left|\displaystyle\int_{I_+^\alpha}\e^{\ii2\theta}d\mu(\theta)\right|^2.
\end{align*}

Similarly, $\langle \jJ, \xX_-^\alpha\circ \xX_-^\alpha \rangle =\dfrac{1}{2}\left(\displaystyle\int_{I_-^\alpha}d\mu(\theta)\right)^2+\dfrac{1}{2}\left|\displaystyle\int_{I_-^\alpha}\e^{\ii2\theta}d\mu(\theta)\right|^2$, and,
\[\langle \jJ, \xX\circ \xX \rangle=\dfrac{1}{2}\left(\displaystyle\int_{\rR / 2\pi \zZ}d\mu(\theta)\right)^2+\dfrac{1}{2}\left|\displaystyle\int_{\rR / 2\pi \zZ}\e^{\ii2\theta}d\mu(\theta)\right|^2.\]

By the equalities above, to prove Lemma~\ref{lem:partition} it suffices to prove the following lemma.
\begin{lemma}
\label{lem:partition-inequality}
For any Borel measure $\mu$ on $\mathbf{R} / 2\pi \mathbf{Z}$, we have
\[
\dfrac{1}{2}\left(\displaystyle\int_{\rR / 2\pi \zZ}d\mu(\theta)\right)^2+\dfrac{1}{2}\left|\displaystyle\int_{\rR / 2\pi \zZ}\e^{\ii2\theta}d\mu(\theta)\right|^2\]
\[\leq \displaystyle\dfrac{C}{2\pi}\int_0^{2\pi}\left(\displaystyle\int_{I_+^\alpha}d\mu(\theta)\right)^2+\left|\displaystyle\int_{I_+^\alpha}\e^{\ii2\theta}d\mu(\theta)\right|^2 d\alpha.
\]
\end{lemma}
\begin{proof}
We focus on measures $\mu$ that are given by a continuous density function, that is $\mu(\theta) = h(\theta)d\theta$ with $h:\rR\to \rR$ continuous, nonnegative and $2\pi$--periodic. By standard limit arguments, this case is sufficient for the proof of the inequality above. We can write $h$ as the Fourier series $h(\theta) = \sum_{n=-\infty}^\infty c_n \e^{\ii n\theta}$ with $c_n=\overline{c_{-n}}\in \Cds$ for every $n$ in $\Zds$ and $\sum_{n=-\infty}^\infty \abs{c_n}^2<\infty$.

Observe that, 
\[
\dfrac{1}{2}\left(\displaystyle\int_{\rR / 2\pi \zZ}h(\theta)d\theta\right)^2+\dfrac{1}{2}\left|\displaystyle\int_{\rR / 2\pi \zZ}\e^{\ii 2\theta}h(\theta)d\theta\right|^2=\dfrac{4\pi^2(c_0^2+\abs{ c_2}^2)}{2}
\]

Now note that,
\begin{align*}
\displaystyle\int_{I_+^\alpha}\e^{\ii 2\theta}h(\theta)d\theta &= \displaystyle\int_{\alpha}^{\alpha+\pi/2}\e^{\ii 2\theta}h(\theta)d\theta+\displaystyle\int_{\alpha+\pi}^{\alpha+3\pi/2}\e^{\ii 2\theta}h(\theta)d\theta \\&=\pi c_{-2}+\displaystyle\sum_{n\neq 0} \dfrac{c_{n-2}(\ii^n-1+(-\ii)^n-(-1)^n)}{\ii n}\e^{\ii n\alpha}\\&=\pi c_{-2}-4\displaystyle\sum_{n\equiv 2 \mod{4}} \dfrac{c_{n-2}}{\ii n}\e^{\ii n\alpha},
\end{align*}
\noindent from which follows that,
\begin{align*}
\displaystyle\dfrac{1}{2\pi}\int_0^{2\pi}\left|\displaystyle\int_{I_+^\alpha}\e^{\ii 2\theta}d\mu(\theta)\right|^2 d\alpha &= \pi^2\abs{c_2}^2+16\displaystyle\sum_{n\equiv 2 \mod{4}}\dfrac{\abs{c_{n-2}}^2}{n^2}\\&\geq \pi^2\abs{c_2}^2+4\abs{c_0}^2
\end{align*}

Similarly,
\begin{align*}
\displaystyle\dfrac{1}{2\pi}\int_0^{2\pi}\left|\displaystyle\int_{I_+^\alpha}d\mu(\theta)\right|^2 d\alpha &= \pi^2\abs{c_0}^2+16\displaystyle\sum_{n\equiv 2 \mod{4}}\dfrac{\abs{c_n}^2}{n^2}\\&\geq \pi^2\abs{c_0}^2+4\lvert c_2\rvert^2
\end{align*}

Putting together these inequalities, we complete the proof.
\end{proof}

\subsection{Proof of Lemma~\ref{lem:chi-vec-doubleprime-upper}}
\label{sec:sublinear}
In this section, we study $\chi''_{\mathrm{vec}, 3}(G)$ and prove Lemma~\ref{lem:chi-vec-doubleprime-upper}.

For convenience, write $\omega = \omega(G)$. Let $\xX$ be any $V \times V$ matrix satisfying the constraints of \eqref{sdp5-mod2}. Then Lemma~\ref{lem:chi-vec-doubleprime-upper} is equivalent to the inequality
$$\sum_{ij \notin E} \xX_{ij}^2 \geq \frac{1}{\omega + 50 \omega^{5/6}} \sum_{i, j\in V} \xX_{ij}^2.$$
We define $\vV, \vv_i, r_i, \theta_i$ as in the previous section. Recall that
$$\xX_{ij} = r_i r_j \cos(\theta_i - \theta_j).$$
For any interval $I \subset \rR / 2\pi \zZ$, let $V_I$ denote the set of $i \in V$ such that $\theta_i \in I$.

First, we show a ``$(\pi / 2 + \epsilon)$" version of the Motzkin-Straus inequality, where the vectors lie in a cone with angle slightly larger than $\pi / 2$.
\begin{lemma}
\label{lem:MS-almost}
Let $\epsilon \in (0, 0.05)$ and $I \subset \rR / 2\pi \zZ$ be an interval of length $(\pi / 2 + 2\epsilon)$. Then we have
$$\sum_{i, j \in V_I, ij \notin E} \xX_{ij}^2 \geq \frac{1}{(1 + 10\sqrt{\epsilon})\omega } \sum_{i, j \in V_I} \xX_{ij}^2.$$
\end{lemma}
To prove this, we observe the following elementary inequalities.
\begin{lemma}
    \label{lem:trig}
    Let $\epsilon \in (0, 0.05)$. Define the map $\phi_\epsilon: [-\epsilon, \pi / 2 + \epsilon] \to [0, \pi / 2]$ by 
    $$\phi_{\epsilon}(x) = \begin{cases}
        0, x \in [-\epsilon, \sqrt{\epsilon}] \\
        \pi / 2, x \in [\pi / 2 - \sqrt{\epsilon}, \pi / 2 + \epsilon] \\
        x, \text{otherwise}
    \end{cases}.$$
    Then for any $x, y \in [-\epsilon, \pi / 2 + \epsilon]$, we have
    $$\cos(x - y)^2 - 2\sqrt{\epsilon} \leq \cos(\phi_\epsilon(x) - \phi_\epsilon(y))^2 \leq (1 + 3\sqrt{\epsilon}) \cos(x - y)^2$$
\end{lemma}
\begin{proof}
    To establish the left inequality, we observe that $\abs{\phi_\epsilon(x) - x} \leq \sqrt{\epsilon}$ for any $x \in [\pi / 2 - \sqrt{\epsilon}, \pi / 2 + \epsilon]$. Thus 
    $$\abs{(x - y) - (\phi_\epsilon(x) - \phi_\epsilon(y))} \leq 2\sqrt{\epsilon}$$
    and by the intermediate value theorem
    $$\abs{\cos(x - y)^2 - \cos(\phi_\epsilon(x) - \phi_\epsilon(y))^2} \leq \sup_x \abs{2 \cos x \sin x} \cdot 2 \sqrt{\epsilon} = 2\sqrt{\epsilon}.$$
    To establish the right inequality, we require some casework. First, assume $x$ lies in $(\sqrt{\epsilon}, \pi / 2 - \sqrt{\epsilon})$. If $y$ also lies in $(\sqrt{\epsilon}, \pi / 2 - \sqrt{\epsilon})$, then $\phi_\epsilon(x) - \phi_\epsilon(y) = x - y$. Otherwise, we may assume that $y$ lies in $[-\epsilon, \sqrt{\epsilon}]$. Then we have $\phi_{\epsilon}(x) - \phi_\epsilon(y) > x - y - \epsilon$, which implies that
    $$\frac{\cos(\phi_\epsilon(x) - \phi_\epsilon(y))^2}{\cos(x - y)^2} = \frac{\sin(\pi / 2 - \phi_\epsilon(x) + \phi_\epsilon(y))^2}{\sin(\pi / 2 - x + y)^2} \leq \left(\frac{\pi / 2 - x + y + \epsilon}{\pi / 2 - x + y}\right)^2.$$
    As $\pi / 2 - x + y \geq \sqrt{\epsilon} - \epsilon$, we conclude that
    $$\frac{\cos(\phi_\epsilon(x) - \phi_\epsilon(y))^2}{\cos(x - y)^2} \leq  \left(\frac{\sqrt{\epsilon}}{\sqrt{\epsilon} - \epsilon}\right)^2 \leq 1 + 3\sqrt{\epsilon}.$$
    Symmetrically, if $y$ lies in $(\sqrt{\epsilon}, \pi / 2 - \sqrt{\epsilon})$, then we also have the desired result. If neither $x$ or $y$ lies in $(\sqrt{\epsilon}, \pi / 2 - \sqrt{\epsilon})$, we may assume that $x \in [-\epsilon, \sqrt{\epsilon}]$. If $y$ lies in $[\pi / 2 - \sqrt{\epsilon}, \pi / 2 + \epsilon]$, then $\cos(\phi_\epsilon(x) - \phi_\epsilon(y))^2 = 0$. If $y$ also lies in $[-\epsilon, \sqrt{\epsilon}]$, then we have
    $$\cos(x - y)^2 \geq \cos(\epsilon + \sqrt{\epsilon})^2 \geq \frac{1}{1 + 3\sqrt{\epsilon}}$$
    as desired.
\end{proof}
\begin{lemma}
    \label{lem:norm}
    We have
    $$\left(\sum_{i \in V} r_i^2\right)^2 \geq \sum_{i, j \in V} \xX_{ij}^2 \geq \frac{1}{2} \left(\sum_{i \in V} r_i^2\right)^2.$$
\end{lemma}
\begin{proof}
    We observe the identity
    $$\sum_{i, j \in V} \xX_{ij}^2 = \dfrac{1}{2}\left(\sum_{j\in V} r_j^2\right)^2+\dfrac{1}{2}\left|\sum_{j\in V} r_j^2\e^{\ii 2\theta_j}\right|^2.$$
    Thus, the left inequality follows by applying the triangle inequality to the second summand, while the right inequality follows since the second summand is non-negative.
\end{proof}
\begin{proof}[Proof of Lemma~\ref{lem:MS-almost}]
    Replacing $G$ with $G[V_I]$, we may assume that $V_I = V$.
    By rotating all $\vv_i$ simultaneously, we may assume that $I = [-\epsilon, \pi/2 + \epsilon]$. 
    
    We define a new set of vectors $\{\ww_i: i \in V\}$ such that $\ww_i$ has norm $r_i$ and argument $\phi_{\epsilon}(\theta_i)$. Let $\wW$ be the Gram-Schmidt matrix of $\{\ww_i\}$. Note that $\ww_i$ lies in quadrant I, so the Motzkin-Straus inequality \eqref{sdp1} implies that
    $$\sum_{ij \notin E} \wW_{ij}^2 \geq \frac{1}{\omega} \sum_{i, j\in V} \wW_{ij}^2.$$
    By Lemma~\ref{lem:trig}, for any $i, j \in V$ we have
    $$\wW_{ij}^2 = r_i^2r_j^2 \cos(\phi_{\epsilon}(\theta_i) - \phi_{\epsilon}(\theta_j))^2 \leq (1 + 3\sqrt{\epsilon}) r_i^2r_j^2\cos(\theta_i - \theta_j)^2 =(1 + 3\sqrt{\epsilon})\xX_{ij}^2.$$
    Thus we have the bound
     $$\sum_{ij \notin E} \xX_{ij}^2 \geq \frac{1}{1 + 3\sqrt{\epsilon}} \sum_{ij \notin E} \wW_{ij}^2.$$
     On the other hand, by Lemma~\ref{lem:trig} we have
     $$\wW_{ij}^2 = r_i^2r_j^2 \cos(\phi_{\epsilon}(\theta_i) - \phi_{\epsilon}(\theta_j))^2 \geq \xX_{ij}^2 - 2\sqrt{\epsilon} r_i^2 r_j^2.$$
     Therefore we have the bound
     $$\sum_{i, j\in V} \wW_{ij}^2 \geq \sum_{i, j\in V} \xX_{ij}^2 - 2\sqrt{\epsilon} \left(\sum_{i \in V}r_i^2\right)^2.$$
     By Lemma~\ref{lem:norm}, we have
     $$\left(\sum_{i \in V_I}r_i^2\right)^2 \leq 2\sum_{i, j\in V} \xX_{ij}^2.$$
     So we get
     $$\sum_{i, j\in V} \wW_{ij}^2 \geq (1 - 4\sqrt{\epsilon}) \sum_{i, j\in V} \xX_{ij}^2.$$
     Combining the above inequalities, we conclude that
     $$\sum_{ij \notin E} \xX_{ij}^2 \geq \frac{1}{1 + 3\sqrt{\epsilon}} \sum_{ij \notin E} \wW_{ij}^2 \geq \frac{1}{(1 + 3\sqrt{\epsilon})\omega} \sum_{i, j\in V} \wW_{ij}^2 \geq \frac{1 - 4\sqrt{\epsilon}}{(1 + 3\sqrt{\epsilon})\omega}\sum_{i, j\in V} \xX_{ij}^2$$
     and the desired result follows.
\end{proof}
We also prove a simple claim about the geometry of vectors.
\begin{lemma}
\label{lem:main-cluster}
For any $\epsilon \in (0, 0.05)$, let $F \subset V \times V$ be the set of pairs $(i, j)$ such that $\cos(\theta_i - \theta_j) < -\epsilon$. Suppose for some $\delta > 0$ we have
$$\sum_{(i, j) \in F} \xX_{ij}^2 \leq \delta \epsilon^2 \sum_{i, j \in V} \xX_{ij}^2.$$
Then there exists an interval $I \subset \rR / 2\pi \zZ$ of length $\pi / 2 + 2\epsilon$ such that
$$\sum_{i \in V_I} r_i^2 \geq (1 - 3\sqrt{\delta}) \sum_{i \in V} r_i^2.$$
\end{lemma}
\begin{proof}
For each $i \in V$, let $F_i$ be the set of $j \in V$ such that $(i, j) \in F$. Observe that $\langle \vv_i, \vv_j \rangle^2 \geq \epsilon^2 r_i^2 r_j^2$ for any $(i, j) \in F$. Thus we have
$$\sum_{(i, j) \in F} r_i^2r_j^2 \leq \delta  \sum_{i, j \in V} \xX_{ij}^2 \leq \delta  \sum_{i, j \in V} r_i^2r_j^2.$$
So we can fix some $i^* \in V$ such that
$$\sum_{j \in F_{i^*}} r_i^2 \leq \delta \sum_{j \in V} r_j^2.$$
Without loss of generality, assume that $\theta_{i^*} = 0$. For each $i \in V \backslash F_{i^*}$, we must have $\cos(\theta_i) \geq -\epsilon$, which implies that $\theta_i \in (-\pi / 2 - 2\epsilon, \pi / 2 + 2\epsilon)$. Let $\alpha$ be the smallest number greater than $-\pi / 2 - 2\epsilon$ such that
$$\sum_{i \in V_{(-\pi / 2 - 2\epsilon, \alpha]}} r_i^2 \geq \sqrt{\delta} \sum_{i \in V} r_i^2.$$
We claim that $I = [\alpha, \alpha + \pi / 2 + 2\epsilon]$ has the desired property. Indeed, set $I_- = (-\pi / 2 - 2\epsilon, \alpha)$, $\overline{I_-} = [-\pi / 2 - 2\epsilon, \alpha]$ and $I_+ = (\alpha + \pi / 2 + 2\epsilon, \pi / 2 + 2\epsilon)$. The definition of $\alpha$ implies
$$\sum_{i \in V_{I_-}} r_i^2 \leq \sqrt{\delta} \sum_{i \in V} r_i^2.$$
Furthermore, for $i \in V_{\overline{I_-}}$ and $j \in V_{I_+}$, we have $\theta_j - \theta_i \in (\pi / 2 + 2\epsilon, 3\pi / 2 - 2\epsilon)$, so $(i, j) \in F$. Therefore, we have
$$\sum_{i \in V_{\overline{I_-}}, j \in V_{I_+}} r_i^2 r_j^2 \leq \delta \sum_{i, j \in V} r_i^2 r_j^2 = \delta \left(\sum_{i \in V} r_i^2\right)^2.$$
So we conclude that
$$\sum_{i \in V_{I_+}} r_i^2 \leq \sqrt{\delta} \sum_{i \in V} r_i^2.$$
Finally, observe that
$$V = F_{i^*} \cup V_{I_-} \cup V_{I_+} \cup V_I.$$
So we have
$$\sum_{i \in V_I} r_i^2 \geq \sum_{i \in V} r_i^2 - \sum_{i \in F_{i^*}} r_i^2 - \sum_{i \in V_{I_-}} r_i^2 - \sum_{i \in V_{I_+}} r_i^2 \geq (1 - 3\sqrt{\delta}) \sum_{i \in V} r_i^2$$
as desired.
\end{proof}
We are now ready to establish the main inequality.
\begin{theorem}
   Given our setup, we have
   $$\sum_{ij \notin E} \xX_{ij}^2 \geq \frac{1}{\omega + 50 \omega^{5/6}} \sum_{i, j\in V} \xX_{ij}^2.$$
\end{theorem}
\begin{proof}
Suppose for the sake of contradiction that
$$\sum_{i, j\in V, ij \notin E} \xX_{ij}^2 < \frac{1}{\omega + 50 \omega^{5/6}} \sum_{i, j\in V} \xX_{ij}^2.$$
By Lemma~\ref{lem:chivec_line}, we have $50 \omega^{5/6} \leq (C - 1) \omega$, where $C \leq 1.5$, so $\omega^{1/6} \geq 30$.

Take $\epsilon = \delta = \omega^{-1/3} \in (0, 0.05)$. Let $F \subset V \times V$ be as in Lemma~\ref{lem:main-cluster}. By assumption, any pair of vertices $(i, j) \in F$ is not an edge in $G$, so 
$$\sum_{(i, j) \in F} \xX_{ij}^2 \leq \sum_{i, j\in V, ij \notin E} \xX_{ij}^2 \leq \frac{1}{\omega} \sum_{i, j\in V} \xX_{ij}^2 = \epsilon^2 \delta \sum_{i, j\in V} \xX_{ij}^2.$$
By Lemma~\ref{lem:main-cluster}, there exists some interval $I \subset \rR / 2\pi \zZ$ of length $\pi/2 + 2\epsilon$ such that
$$\sum_{i \in V_I} r_i^2 \geq (1 - 3\omega^{-1/6}) \sum_{i \in V} r_i^2.$$
By Lemma~\ref{lem:MS-almost}, we have
$$\sum_{i, j\in V, ij \notin E} \xX_{ij}^2 \geq \frac{1}{(1 + 10 \omega^{-1/6})\omega} \sum_{i, j \in V_I} \xX_{ij}^2.$$
Furthermore, observe that
$$\sum_{i, j \in V} \xX_{ij}^2 - \sum_{i, j \in V_I} \xX_{ij}^2 \leq 2\sum_{i \in V \backslash V_I} r_i^2 \sum_{j \in V} r_j^2 \leq 6\omega^{-1/6} \left(\sum_{i \in V} r_i^2\right)^2$$
By Lemma~\ref{lem:norm}, we have
$$\sum_{i, j \in V} \xX_{ij}^2 - \sum_{i, j \in V_I} \xX_{ij}^2 \leq 12 \omega^{-1/6} \sum_{i, j \in V} \xX_{ij}^2.$$
So we conclude that
$$\sum_{i, j\in V, ij \notin E} \xX_{ij}^2 \geq \frac{1 - 12 \omega^{-1/6}}{(1 + 10 \omega^{-1/6})\omega } \sum_{i, j \in V} \xX_{ij}^2 > \frac{1}{(1 + 50 \omega^{-1/6})\omega } \sum_{i, j \in V} \xX_{ij}^2$$
contradiction.
\end{proof}
\subsection{Lower Bounds}
\label{sec:counterexamples}

In this section, we prove our lower bounds on $\chi_{\mathrm{vec}, 3}(G)$ and $\chi_{\mathrm{vec}, 3}'(G)$.

We start by explicitly showing that $\chi_{\mathrm{vec}, 3}'(C_5) > 2$.
\begin{lemma}
We have $\chi_{\mathrm{vec}, 3}'(C_5) > 2.099$.
\end{lemma}
\begin{proof}
    Let
    $$\vV = \begin{bmatrix}
        \cos(0) & \sin(0) \\
        \cos(\pi / 5) & \sin(\pi / 5) \\
        \cos(2\pi / 5) & \sin(2\pi / 5) \\
        \cos(3\pi / 5) & \sin(3\pi / 5) \\
        \cos(4\pi / 5) & \sin(4\pi / 5) \\
    \end{bmatrix}$$
    and
    $$\xX = \vV \vV^T.$$
    Observe that $\xX$ is rank $2$ and positive semidefinite. Its entries satisfy
    $$\abs{\xX_{ij}} = \begin{cases}
        1, \text{ if }i = j, \\
        \cos(\pi / 5), \text{ if }i - j \in \{\pm 1\},\\
        \cos(2\pi / 5), \text{ if } i - j \in \{\pm 2\}.
    \end{cases}$$
    where the indices are interpreted modulo $5$. Thus we have
    $$\chi_{\mathrm{vec}, 3}'(C_5) \geq \frac{\langle \jJ, \xX\circ\xX \rangle}{\langle \iI+\overline{\aA}, \xX\circ \xX\rangle} = 1 + \frac{10 \cos(\pi / 5)^2}{5 + 10 \cos(2\pi / 5)^2} > 2.099,$$
    as desired.
\end{proof}
Now we apply a blowup operation on $G = C_5$ to construct the desired $G$ for every even $\omega$. For a graph $G = (V, E)$, let $G^n$ denote the graph with vertex set $V^n = V \times [n]$, and edge set
$$E^n = \{\{(v, i), (w, j)\}: \{v, w\} \in E \text{ or } i \neq j\}.$$ 
We observe that:
\begin{proposition}
For any graph $G$, we have $\omega(G^n) = n \omega(G).$
\end{proposition}
\begin{proof}
 On one hand, if $C$ is a clique in $G$, then $C \times [n]$ is a clique in $G^n$, so $\omega(G^n) \geq n \omega(G)$. On the other hand, let $C$ be a maximum clique in $G^n$. Let $C_i$ be the vertices in $C$ with second coordinate equal to $i$. Then $C_i$ must form a clique in $G$, so we have
 $$\omega(G^n) = \abs{C} = \sum_{i = 1}^n \abs{C_i} \leq n \omega(G)$$
 as desired.
\end{proof}
\begin{proposition}
For any graph $G$, we have
$\chi_{\mathrm{vec}, 3}'(G^n) \geq n \chi_{\mathrm{vec}, 3}'(G)$. 
\end{proposition}
\begin{proof}
    Let $\xX$ be any matrix satisfying the condition of the optimization problem defining $\chi_{\mathrm{vec}, 3}'(G)$. Let $\xX^n = \frac{1}{\sqrt{n}}\xX \otimes \jJ_n$. More explicitly, $\xX^n$ is indexed by $V \times [n]$, and the element at the $(v, i)$-th row and $(w, j)$-th column of $\xX^n$ is equal to $\frac{1}{\sqrt{n}}\xX_{v, w}$. By the property of the Kronecker product, $\xX^n$ is positive semidefinite and $\mathrm{rank}~\xX^n = \mathrm{rank}~\xX \cdot \mathrm{rank}~\jJ_n = 2$. Furthermore, we have
    $$\langle \iI + \bar{\aA}_{G^n}, \xX^n \circ \xX^n \rangle = \frac{1}{n} \cdot n \langle \iI + \bar{\aA}_{G}, \xX \circ \xX \rangle = 1.$$
    Thus we have
    $$\chi_{\mathrm{vec}, 3}'(G^n) \geq \langle \jJ, \xX^n \circ \xX^n \rangle = n \langle \jJ, \xX \circ \xX \rangle$$
    as desired.
\end{proof}

\begin{proof}[Proof of Proposition~\ref{prop:chi-3-lower}]
    Let $G = C_5^{\omega / 2}$. By the preceding lemmas, we have $\omega(G) = \omega / 2 \cdot \omega(C_5) = \omega$ and $\chi_{\mathrm{vec}, 3}'(G) \geq \omega \chi_{\mathrm{vec}, 3}'(C_5) / 2 \geq 2.099 \omega / 2 \geq 1.04 \omega$.
\end{proof}

\subsection{A quick proof for the regular case} \label{sec:quickregular}

For completeness, we include a concise proof of Conjecture~\ref{conj:bollobas_nikiforov} 
for regular graphs using our matrix language, based on the proof by the third author in \cite{zhang2024first}.

We assume the graph $G$ is not complete, so the two largest eigenvalues $\lambda_1$ and $\lambda_2$ are nonnegative. The proof starts similarly as the proof of Theorem~\ref{thm:main_result2}.
We define the matrices
	\[
		\xX = \lambda_1 (\vv_1\vv_1^\T) + \lambda_2 (\vv_2 \vv_2^\T)
	\]
and
\[
	\yY = \vv_1\vv_2^\T + \vv_2\vv_1^\T.
\]
It follows that 
\begin{align*}
(\xX \circ \xX) + \lambda_1 \lambda_2 (\yY \circ \yY) & = (\lambda_1 (\vv_1 \circ \vv_1) + \lambda_2 (\vv_2 \circ \vv_2))(\lambda_1 (\vv_1 \circ \vv_1) + \lambda_2 (\vv_2 \circ \vv_2))^\T \\
& + 4 \lambda_1 \lambda_2 (\vv_1 \circ \vv_2)(\vv_1 \circ \vv_2)^\T.
\end{align*}
Therefore this matrix is positive semidefinite, nonnegative, and has rank $2$. So it is completely positive, and we may apply \eqref{eq:ms} with $\zZ = (\xX \circ \xX) + \lambda_1 \lambda_2 (\yY \circ \yY)$. We obtain
\[
	\langle \aA,\xX \circ \xX \rangle + \lambda_1\lambda_2 \langle \aA,\yY \circ \yY\rangle \leq \left(1 - \frac{1}{\omega(G)} \right) \Big( \langle \jJ , \xX \circ \xX\rangle +  \lambda_1\lambda_2 \langle \jJ , \yY \circ \yY \rangle \Big). 
\]
Recall that $\langle \aA,\xX\rangle^2 \leq 2m \langle \aA , \xX \circ \xX \rangle$, so if
\[
	\langle \aA,\xX \circ \xX \rangle \leq \left(1 - \frac{1}{\omega(G)} \right) \langle \jJ , \xX \circ \xX\rangle,
\] 
then the conjecture would follow. So we assume otherwise, which implies
\[
	\langle \aA,\yY \circ \yY\rangle < \left(1 - \frac{1}{\omega(G)} \right)\langle \jJ , \yY \circ \yY \rangle = 2 \left(1 - \frac{1}{\omega(G)} \right).
\]
We now assume $G$ is $k$-regular, thus $\vv_1$ is the all $1$s vectors normalized and $\lambda_1 = k$. Hence
\[
	\frac{1}{2}\langle \aA , \yY \circ \yY \rangle = \langle \aA , (\vv_1 \circ \vv_1)(\vv_2 \circ \vv_2)^\T + (\vv_1 \circ \vv_2)(\vv_1 \circ \vv_2)^\T \rangle = \frac{k + \lambda_2}{n}.
\]
Therefore
\[
	\lambda_1^2 + \lambda_2^2 \leq k (k + \lambda_2) \leq \left(1 - \frac{1}{\omega(G)} \right) k n = \left(1 - \frac{1}{\omega(G)} \right) 2m,
\]
as we wanted.

\section{Vertex weights} \label{sec:vertexweights}

The well known theory of the polyhedra \texttt{STAB} and \texttt{QSTAB}, and the convex corner \texttt{TH}, relies on the extension of the parameters $\alpha$, $\vartheta$ and $\chi_f$ to their vertex weighted versions (see for instance \cite{LovaszGrotschelSchrijverTheta}). The formulation for the vertex weighted version of $\chi_\mathrm{vec}$ goes back to \cite{KnuthSandwich}.

Here we adopt the formulations for the vertex weighted versions of $\omega$ and $\chi_\mathrm{vec}$ introduced in \cite{MarcelAxiomatic}, and prove that both admit a formulation analogous to \eqref{sdp2}. Let $\ww \geq \0$ be a nonnegative vector in $\Rds^n$. Denote (as usual) by $\sqrt{\ww}$ the vector in $\Rds^n$ obtained from $\ww$ by taking the entrywise square root. Recall that if $G$ is fixed, then we denote by $\aA$ its adjacency matrix and by $\ov{\aA}$ the adjacency matrix of its complement graph.

\begin{theorem}[Theorems 5 and 14, Proposition 16, in \cite{MarcelAxiomatic}]
	 Let $G$ be a graph and $\ww \in \Rds^n_+$. Consider the program	
	\begin{align} \max \quad & \langle \sqrt{\ww}\sqrt{\ww}^\T, \zZ\rangle, \label{sdp3}\\
\textrm{subject to} \quad & \langle \iI, \zZ\rangle = 1, \nonumber\\
& \zZ \circ \overline{\aA} = 0, \nonumber\\
& \zZ \in \Kc \nonumber
\end{align}
If $\Kc$ is the doubly nonnegative cone, this program is equal to $\chi_{\mathrm{vec}}(G;\ww)$; and if $\Kc = \mathcal{CP}_n$, this program is equal to $\omega(G;\ww)$.
\end{theorem}

We show below how to adapt the trick from Lemma~\ref{lem:chivec} to these formulations. If $\ww \in \Rds^n_+$, denote by $\dD(\ww)$ the diagonal matrix whose diagonal is equal to $\ww$, and if $\mM$ is a symmetric matrix, we introduce the notation
\[
	\mM_{\ww} = \frac{\dD(\sqrt{\ww})~\mM~\dD(\sqrt{\ww})^{+} + \dD(\sqrt{\ww})^{+}~\mM~\dD(\sqrt{\ww})}{2},
\]
where $\dD(\ww)^{+}$ is the pseudo-inverse of $\dD(\ww)$.

\begin{theorem}\label{thm:weighted}
Let $G$ be a graph and $\ww \in \Rds^n_+$. Consider the program	
\begin{align} \max \quad & \langle \sqrt{\ww}\sqrt{\ww}^\T, \zZ\rangle \label{sdp4}\\
\textrm{subject to} \quad & \langle (\iI+ \overline{\aA})_\ww, \zZ \rangle = 1 \nonumber \\
& \zZ \in \Kc. \nonumber
\end{align}
If $\Kc$ is the doubly nonnegative cone, this program is equal to $\chi_{\mathrm{vec}}(G;\ww)$; and if $\Kc = \mathcal{CP}_n$, this program is equal to $\omega(G;\ww)$.\end{theorem}
\begin{proof}
As we did in Lemma~\ref{lem:chivec}, our goal is to show that for both the doubly nonnegative and the completely positive cones, the program \eqref{sdp4} is equivalent to \eqref{sdp3}. 

If $\zZ_0$ is optimum for \eqref{sdp3}, it follows that rows and columns of $\zZ_0$ corresponding $0$ entries of $\ww$ are equal to $\0$, and in this case it is clearly a feasible solution for \eqref{sdp4} with the same objective value.

If $\zZ_0$ is an optimum solution for \eqref{sdp4}, we may assume that $\zZ_0$ has null rows and columns where $\ww$ has $0$ entries. Let $S$ denote the set of edges $uv$ in $E(\overline{G})$ for which $(\zZ_0)_{uv} > 0$, and $\ee_u$ the characteristic vector of vertex $u$. Then take
\[
\widetilde{\zZ}_0:=\zZ_0 + \dD(\sqrt{\ww})^{+}\left(\displaystyle\sum_{uv\in S}(\zZ_0)_{uv}(\ee_u-\ee_v)(\ee_u-\ee_v)^\T\right)\dD(\sqrt{\ww})^{+}.
\]
Note that
\begin{enumerate}[(a)]
	\item $\langle ~\sqrt{\ww}\sqrt{\ww}^\T~,~\displaystyle\sum_{uv\in S}(\zZ_0)_{uv}\dD(\sqrt{\ww})^{+}(\ee_u-\ee_v)(\ee_u-\ee_v)^\T \dD(\sqrt{\ww})^{+}~\rangle = 0$, thus
	\[\langle \sqrt{\ww}\sqrt{\ww}^\T,\zZ_0 \rangle = \langle \sqrt{\ww}\sqrt{\ww}^\T, \widetilde{\zZ}_0\rangle. \]
	\item $\widetilde{\zZ}_0 \circ \ov{\aA} = 0$,
	\item Looking at vertices $u$ and $v$ with $uv \in E(\ov{G})$ where $\ww$ is nonzero, note that
	\[
		\frac{1}{\ww_u} + \frac{1}{\ww_v} - \frac{1}{\sqrt{\ww_u\ww_v}}\left(\frac{\sqrt{\ww_u}}{\sqrt{\ww_v}} + \frac{\sqrt{\ww_v}}{\sqrt{\ww_u}}\right) = 0,
	\]
	so $\langle~(\iI + \ov{\aA})_\ww~,~\displaystyle\sum_{uv\in S}(\zZ_0)_{uv}\dD(\sqrt{\ww})^{+}(\ee_u-\ee_v)(\ee_u-\ee_v)^\T \dD(\sqrt{\ww})^{+}~\rangle = 0$, therefore, by (b),
	\[ \langle(\iI + \ov{\aA})_\ww, \zZ_0 \rangle = \langle \iI , \widetilde{\zZ}_0 \rangle\]
	\item If $\zZ_0$ is doubly nonnegative, then so is $\widetilde{\zZ}_0$ (this is quite straightforward),
	\item If $\zZ_0$ is completely positive, then so is $\widetilde{\zZ}_0$ (this follows from a known result, see for instance \cite[Lemma 3.36]{shaked2021copositive}).
\end{enumerate}

Therefore $\widetilde{\zZ}_0$ is a feasible solution for \eqref{sdp3} with the same objective value as $\zZ_0$ in \eqref{sdp4}.	
\end{proof}

In \cite[Theorem 5]{gibbons1997continuous}, a version of the Motzkin-Straus characterization for the weighted clique number was presented. As a consequence of Theorem~\ref{thm:weighted}, we can very easily recover this formulation.

\begin{corollary}
	Given a graph $G$ and for any $\ww \geq \vec 0$, if
	\[
		\vec{W} = \dD(\ww)^{+} + \frac{1}{2}(\ov{\aA}~ \dD(\ww)^+ + \dD(\ww)^+ ~\ov{\aA}),
	\]
	then
	\[
		\frac{1}{\omega(G;\ww)} = \min \{\vec v^\T \vec{W} \vec v : \1^\T \vec v = 1, \vec v \geq \vec 0.\}
	\] 
\end{corollary}
\begin{proof}
	If follows easily from \eqref{sdp4} that,
	\[
		\frac{1}{\omega(G;\ww)} = \min \{\vec v^\T ~(\iI+\ov{\aA})_\ww~\vec v : \sqrt{\ww}^\T \vec v = 1, \vec v \geq \vec 0\}.
	\] 
	The result is now obtained upon making $\uu = \dD(\sqrt{\ww})^+ \vv$.
\end{proof}

From our work in Section~\ref{sec:2}, another consequence of Theorem~\ref{thm:weighted} is that, if $\xX \in \Kc$ and $\nu_\Kc(G;\ww)$ is defined as the optimum in \eqref{sdp4}, then
\begin{align}
	\langle \aA_\ww , \xX \rangle \leq \langle \jJ_\ww , \xX \rangle - \frac{1}{\nu_\Kc(G;\ww)} \langle \sqrt{\ww} \sqrt{\ww}^\T,\xX\rangle.
\end{align}
For instance, choosing $\vv$ and $\ww$ so that $\vv \circ \ww$ is the eigenvector corresponding to the largest eigenvalue of $\aA$, and making $\xX = \dD(\sqrt{\ww}) \vv \vv^\T \dD(\sqrt{\ww})$, this implies
\begin{align}
	\lambda_1(\aA) \langle \ww , \vv^{\circ 2}\rangle \leq \langle \1,\vv\rangle \langle \ww , \vv \rangle - \frac{1}{\omega(G;\ww)} \langle \ww , \vv \rangle^2,
\end{align}
which can be used to derive eigenvalue bounds for $\omega(G;\ww)$.

\section{Final remarks}
\subsection{Open problems}
By Lemma~\ref{lem:chivec_line_weak} and Proposition~\ref{prop:chi-3-lower}, for every even $\omega$ we have
$$1.04 \omega \leq \sup_{G: \omega(G) = \omega}\chi_{\mathrm{vec}, 3}'(G) \leq \sup_{G: \omega(G) = \omega} \chi_{\mathrm{vec}, 3}(G) \leq 4\omega.$$ 
This prompts the question of determining the optimal coefficient of the linear relation.
\begin{problem}
    Determine the optimal constant $C$ such that
    $\chi_{\mathrm{vec}, 3}(G) \leq C\omega(G)$
    holds for all graphs $G$. Do the same for $\chi_{\mathrm{vec}, 3}'(G)$.    
\end{problem}
In contrast, Lemma~\ref{lem:chi-vec-doubleprime-upper} shows that $\chi_{\mathrm{vec}, 3}''(G)$ is bounded above by $\omega(G) + o(\omega(G))$, while it is bounded below by $\omega(G)$. Based on this behavior, we conjecture that $\chi_{\mathrm{vec}, 3}''(G)$ is equal to $\omega(G)$, at least when $\omega(G)$ is large.
\begin{conjecture}
    For sufficiently large $\omega$ and all graphs $G$ with clique number $\omega$, we have $$\chi_{\mathrm{vec}, 3}''(G) = \omega.$$
\end{conjecture}
If this conjecture turns out to be false, it is interesting to determine a tight upper bound for $\chi_{\mathrm{vec}, 3}''(G) - \omega$.

By unpacking the definition of $\chi_{\mathrm{vec}, 3}''(G)$, the conjecture can be rephrased as a weighted analog of Tur\'{a}n's theorem. We record this statement below.
\begin{conjecture}
For sufficiently large $\omega$, the following holds. Let $G = (V, E)$ be a graph with clique number $\omega$. For each vertex $i$, we assign a two-dimensional vector $\vv_i \in \rR^2$. For each pair of vertices $(i, j)$, we assign the ``edge weight" $$\ww_{i, j} = \langle \vv_i, \vv_j \rangle^2.$$ 
If the vectors satisfy $\langle \vv_i, \vv_j \rangle \geq 0$ for every $\{i, j\} \in E$, then we have
$$\sum_{\{i, j\} \in E} \ww_{i, j} \leq \frac{1}{2}\left(1 - \frac{1}{\omega}\right) \sum_{i, j \in V} \ww_{i, j}$$
where the sum on the left hand side goes over each edge once.
\end{conjecture}

It is not hard to check that if $\vv_i$ are one-dimensional, then this is equivalent to the Motzkin-Straus inequality. Furthermore, Proposition~\ref{prop:chi-3-lower} shows that the condition $\langle \vv_i, \vv_j \rangle \geq 0$ for every $\{i, j\} \in E$ is necessary. We may try any of the numerous proofs of the Motzkin-Straus inequality and Tur\'{a}n's theorem on this problem (see e.g. \cite{gtac, weightedturan}). Unfortunately, we could not make any of them work.


\subsubsection*{Acknowledgements}

G. Coutinho and T. Spier acknowledge research grants from FAPEMIG and CNPq, and conversations about the topic of this paper with Marcel K. de Carli Silva, Thiago Oliveira and Levent Tunçel.

S. Zhang thanks Stanford University for supporting his research with the Craig Franklin Fellowship in Mathematics. S. Zhang thanks Prof. Hong Liu, Ting-Wei Chao, Zixiang Xu, Zhuo Wu, Jun Gao, and many others for productive discussions at the Institute for Basic Science in Korea. Shengtong also thanks Dr. Jonathan Tidor, Nitya Mani and Anqi Li for discussions and encouragements.
\bibliographystyle{plain}
\IfFileExists{references.bib}
{\bibliography{references.bib}}
{\bibliography{../references}}

	
\end{document}